\documentclass[twoside,12pt,leqno]{article}

\usepackage{color}

\usepackage[utf8]{inputenc}
\usepackage{amsthm,amssymb,physics,latexsym,amsmath,amscd,array,graphicx,enumerate,lmodern,amsfonts,mathrsfs,mathtools,mparhack,lipsum,xcolor,tikz-cd,tikz,tikzscale,pigpen,multicol,capt-of}  
\usepackage[all]{xy}
\usepackage[colorlinks, citecolor=black, urlcolor = black, linkcolor=black]{hyperref}

\usetikzlibrary{matrix,arrows,shadows,backgrounds,positioning,decorations.markings,calc}

\newcommand{\effv}{\mathrm{\mathsf{effv}}}

\newcommand{\effl}{\mathrm{\mathsf{effl}}}
\newcommand{\TC}{\mathrm{\mathsf{TC}}}
\newcommand{\imm}{\mathrm{\mathsf{Imm}}}
\newcommand{\T}{\operatorname{T}}
\newcommand{\K}{\operatorname{K}}

\newcommand{\secat}{\mathrm{\mathsf{secat}}}
\newcommand{\A}{\mathrm{\mathsf{a}}}
\newcommand{\B}{\mathrm{\mathsf{b}}}
\newcommand{\R}{\mathbb{R}}
\newcommand{\Z}{\mathbb{Z}}

\newtheorem{theorem}{Theorem}[section]
\newtheorem{definition}[theorem]{Definition}
\newtheorem{lemma}[theorem]{Lemma}
\newtheorem{proposition}[theorem]{Proposition}

\newtheorem{remark}[theorem]{Remark}
\newtheorem{corollary}[theorem]{Corollary}

\title{Effectual Topological Complexity}
\author{Natalia Cadavid-Aguilar\footnote{The first author is grateful for support from  Fundaci\'on Sof\'ia Kovalévskaia and from FORDECYT grant 265667 ``Programa para un avance global e integral de la matem\'atica mexicana''.}, Jes\'us Gonz\'alez, B\'arbara Guti\'errez\thanks{Supported by the project 20201646 of the Secretar\'{i}a de Investigaci\'on y Posgrado at the IPN.}\\ and Cesar A.~Ipanaque-Zapata\footnote{The author would like to thank grant\#2016/18714-8, S\~{a}o Paulo Research Foundation (FAPESP) for financial support.}}

\date{}

\begin{document}

\maketitle

\begin{abstract}
We introduce the effectual topological complexity (ETC) of a $G$-space $X$. This is a $G$-equivariant homotopy invariant sitting in between the effective topological complexity of the pair $(X,G)$ and the (regular) topological complexity of the orbit space $X/G$. We study ETC for spheres and surfaces with antipodal involution, obtaining a full computation in the case of the torus. This allows us to prove the vanishing of twice the non-trivial obstruction responsible for the fact that the topological complexity of the Klein bottle is~$4$. In addition, this gives a counterexample to the possibility ---suggested in Pave\v{s}i\'c's work on the topological complexity of a map--- that ETC of $(X,G)$ would agree with Farber's $\TC(X)$ whenever the projection map $X\to X/G$ is finitely sheeted. We conjecture that ETC of spheres with antipodal action recasts the Hopf invariant one problem, and describe (conjecturally optimal) effectual motion planners.
\end{abstract}



\section{Main results}
For a group $G$ and a $G$-space $X$, we define the \emph{effectual topological complexity} (ETC) $\TC_{\effl}^G(X)$ as the sectional category of the fibration $\epsilon: PX \rightarrow X \times (X/G)$ given by $\epsilon(\gamma)=(\gamma(0),[\gamma(1)])$, where square brackets stand for $G$-orbits. This definition is motivated by and closely related to (though different from) Błaszczyk-Kaluba's effective topological complexity (etc) $\TC_{\effv}^{G}(X)$, revisited in  Section~\ref{secrevi}. In addition, for nicely behaved $G$-spaces, ETC is a special case of Pave\v{s}i\'c's $\TC(\pi)$, the topological complexity of the projection map $\pi\colon X\to X/G$. But more importantly, ETC serves as a connecting link between $\TC_{\effv}^{G}(X)$ and Farber's topological complexity of the orbit space, $\TC(X/G)$:

\begin{theorem}\label{principalA}
If $X$ is Hausdorff and $G$ is a discrete group acting properly discontinuously on $X$, then
\begin{equation}\label{sandwich}
\TC_{\effv}^{G}(X)\leq\TC_{\effl}^G(X)\leq\TC(X/G).
\end{equation}
\end{theorem}

Both etc and ETC are motivated by the idea of taking advantage of symmetries in the motion planning of an autonomous system (see Section~\ref{seceffel}). Yet the two concepts are essentially different from each other, as indicated next.

The case of an $n$-dimensional sphere ($n\geq1$) with antipodal involution $(S^n,\mathbb{Z}_2)$ and quotient space $P^n=S^n/\mathbb{Z}_2$, the $n$-dimensional real projective space, is particularly interesting. It is easy to see that $\TC_{\effv}^{\mathbb{Z}_2}(S^n)=1$ (Corollary~\ref{hduiwydg}). More difficult to prove~is
$$
\TC_{\effl}^{\mathbb{Z}_2}(S^n)\in\{n,n+1\}\mbox{ \ \,and \ \,}\TC(S^n/\mathbb{Z}_2)=\imm(P^n)-\begin{cases}1,&n=1,3,7;\\0,&\mbox{otherwise.}\end{cases}
$$
Here $\imm(P^n)$ stands for the smallest Euclidean dimension where $P^n$ admits a smooth immersion. The values above for $\TC_{\effv}^{\mathbb{Z}_2}(S^n)$ and $\TC(S^n/\mathbb{Z}_2)$ come from~\cite{BK} and \cite{FTY}, respectively, while the estimate for $\TC_{\effl}^{\mathbb{Z}_2}(S^n)$ comes from~\cite[Proposition 4.7]{Pavesic}. The latter estimate should be compared to a folk empirical belief that the (currently unknown) value of $\imm(P^n)$ would have the form $2n-2\alpha(n)+o(\alpha(n))$, where $\alpha(n)$ denotes the number of ones in the dyadic expansion of $n$.

Our current knowledge of $\imm(P^n)$ (summarized in~\cite{davistable}) gives that, for antipodally acted spheres, both inequalities in~(\ref{sandwich}) are strict as long as $n\geq8$. In contrast, in the parallelizable sphere case ($n=1,3,7$), $\TC_{\effl}^{\mathbb{Z}_2}(S^n)=\TC(S^n/\mathbb{Z}_2)=n$, for in fact $\imm(P^n)=n+1$. We believe that these three special values of~$n$ are the only ones having $\TC_{\effl}^{\mathbb{Z}_2}(S^n)=n$, which would yield a new form of the Hopf invariant one problem. Such a possibility should be compared to the (homotopically more accessible) fact in~\cite{MR3762837} that $\TC(S^n)$ is minimal possible precisely when the classical Hopf invariant of the Whitehead square of the identity on $S^n$ vanishes (i.e., when $n$ is odd).

As with spheres, the behavior of the inequalities in~(\ref{sandwich}) is very subtle in the case of an orientable closed surface with antipodal involution. Explicitly, think of the orientable surface $\Sigma_g$ of genus $g$ embedded in $\R^3$ as shown in Figure~\ref{embedded}, so that reflections in the $xy$-, $yz$-, and $xz$-planes yield symmetries of $\Sigma_g$.
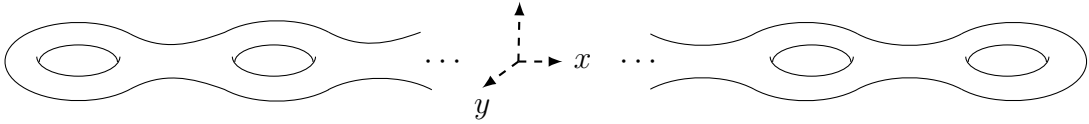
\begin{figure}[h!]
\begin{center}
\begin{tikzpicture}[scale=0.65]
\node at (7.5,0,0){$\dotsb$};
\node at (11.5,0,0){$\dotsb$};
\draw[thick,-latex, dashed] (9,0,0) -- (9.9,0,0) node[right]{$x$};
\draw[thick,-latex, dashed] (9,0,0) -- (9.8,1,4) node[below]{$y$};
\draw[thick,-latex, dashed] (9,0,0) -- (9.8,2,2.01);
\draw (0.8,-0.01) arc (0:180:0.8 and 0.35); 
\draw (-0.85,0.1) arc (180:360:0.85 and 0.4); 
\draw (4.8,-0.01) arc (0:180:0.8 and 0.35); 
\draw (3.15,0.1) arc (180:360:0.85 and 0.4); 
\draw (15.8,-0.01) arc (0:180:0.8 and 0.35); 
\draw (14.15,0.1) arc (180:360:0.85 and 0.4); 
\draw (19.8,-0.01) arc (0:180:0.8 and 0.35); 
\draw (18.15,0.1) arc (180:360:0.85 and 0.4); 
\draw (1.0605,0.5656) arc (45:315:1.5 and 0.8); 
\draw (1.0605,0.5656) to[out=-28.1,in=200] (3,.6); 
\draw (5.12,0.60) arc (45:135:1.5 and 0.8); 
\draw (5.12,0.6) to[out=-28.1,in=200] (7,.6);
\draw (1.0605,-0.5656) to[out=28.1,in=160] (3,-0.5756); 
\draw (2.998,-0.576) arc (225:315:1.5 and 0.8); 
\draw (7.23,-0.578) arc (45:135:1.5 and 0.8); 
\draw (18.0605,-0.5656) arc (-135:135:1.5 and 0.8);
\draw (18.06,-0.566) arc (45:135:1.5 and 0.8); 
\draw (13.82,-0.567) arc (225:315:1.5 and 0.8); 
\draw (13.82,-0.566) arc (45:135:1.5 and 0.8); 
\draw (15.94,0.565) arc (225:315:1.5 and 0.8); 
\draw (15.94,0.566) arc (45:135:1.5 and 0.8); 
\draw (11.70,0.566) arc (225:315:1.5 and 0.8);
\end{tikzpicture}
\end{center}
\caption{Embedding $\Sigma_g\subset\mathbb{R}^3$}
\label{embedded}
\end{figure}
Let $\sigma$ stand for the ``antipodal" (orientation-reversing) involution on $\Sigma_g$ given by $\sigma(x,y,z)=(-x,-y,-z)$. This makes $\Sigma_g$ into a $\Z_2$-space with quotient $N_{g+1}$, the nonorientable surface of genus $g+1$. The case $g=1$ is similar to the situation noted above for $S^n$ with $n\geq8$, in that:

\begin{theorem}\label{PrincipalB}
Both inequalities in~(\ref{sandwich}) are strict for $(X,G)=(\T,\mathbb{Z}_2)$, the 2-torus with antipodal involution. In fact
\begin{equation}\label{cfto}
\TC_{\effv}^{\mathbb{Z}_2}(\T)=2, \hspace{2mm}\TC_{\effl}^{\mathbb{Z}_2}(\T)=3\mbox{ \ and \ }\TC(\T/\mathbb{Z}_2)=4.
\end{equation}
\end{theorem} 

The values for $\TC_{\effv}^{\mathbb{Z}_2}(\T)$ and $\TC(\T/\mathbb{Z}_2)$ come from~\cite[paragraph following Theorem~1.1]{JN} and \cite{CV}, respectively. Unlike the first equality in~(\ref{cfto}), the proof argument (in Section~\ref{sectoro}) for the second equality in~(\ref{cfto}) is far from being elementary; yet it is not as technically involved as Cohen-Vandembroucq's proof of the third equality in~(\ref{cfto}). On the other hand, Theorem~\ref{PrincipalB} is somehow singular among surfaces for, just as for parallelizable spheres, at least one of the inequalities in~(\ref{sandwich}) is an equality in the case of a larger genus surface. Indeed, Theorem~\ref{principalA}, \cite{CV} and~\cite[Theorem~1.1]{JN} yield
\begin{equation}\label{supersandwich}
3\leq\TC_{\effv}^{\mathbb{Z}_2}(\Sigma_g)\leq \TC_{\effl}^{\mathbb{Z}_2}(\Sigma_g)\leq\TC(\Sigma_g/\mathbb{Z}_2)=4, \quad\mbox{for $g\geq2$.}
\end{equation}
We have not been able to construct a $4$-domains effectual motion planner for $\Sigma_2$, neither in the effective nor in the effectual realms, and it might of course be plausible that all the numbers in~(\ref{supersandwich}) equal $4$. Such a possibility could be addressed via obstruction theory, though the explicit details would seem to be even more challenging than those carried out for $g=1$ in~\cite{CV} (note that $\Sigma_1/\mathbb{Z}_2$ is the Klein bottle). 

\section{Effective topological complexity revisited}\label{secrevi}
Let $PX$ stand for the free-path space of a topological space $X$. Recall that Farber's topological complexity $\TC(X)$ is the sectional category of the end-points evaluation map $e_{0,1}\colon PX\to X\times X$. (We use sectional category of a map $f\colon E\to B$ in the reduced sense, i.e., one less than the minimal number of open sets covering $B$ and on each of which $f$ admits a homotopy local section.) Let $G$ be a topological group acting on the right on $X$. For an integer $k\geq 2$, let $P_k^G(X)$ be the subspace of $(PX \times G)^{k-1} \times PX$ consisting of the tuples $(\alpha_1, g_1, \ldots, \alpha_{k-1}, g_{k-1}, \alpha_k)$ such that $\alpha_i(1)\cdot g_i=\alpha_{i+1}(0)$. Note that $P_k^G(X)$ sits inside $P_{k+1}^G(X)$ as a subspace retract, with inclusion $\iota: P_k^G(X) \rightarrow P_{k+1}^G(X)$ and retraction $r:P_{k+1}^G(X) \rightarrow P_k^G(X)$ given by 
\begin{align*}
\iota(\alpha_1, g_1, \ldots,g_{k-1}, \alpha_k)&=(\alpha_1, g_1, \ldots, g_{k-1}, \alpha_k, e, \overline{\alpha_k(1)}),\\r(\alpha_1, g_1, \ldots, g_{k}, \alpha_{k+1})&=(\alpha_1, g_1, \ldots, g_{k-1}, \alpha_k),
\end{align*}
where $\overline{\alpha_k(1)}$ stands for the path with constant value $\alpha_k(1)$.

\begin{figure}[h]
    \centering
   \begin{tikzpicture}
   \begin{scope}[thin, decoration={
    markings,
    mark=at position 0.5 with {\arrow{latex}}}] 
   \draw[postaction={decorate},thick] (0, 0) .. controls(1,2) and (2, 1.2) .. (3, 2);
   \draw[postaction={decorate},thick] (2, 0) .. controls(3,1.3) and (4, 0.5) .. (5, 2);
   \draw[postaction={decorate},thick] (4, 0) .. controls(4.3,1.5) and (6.3, 0.4) .. (7, 2);
   \draw (.9,1.2) node[above left] {$\alpha_1$};
   \draw (3.9,1.2) node[above left] {$\alpha_2$};
   \draw (6,1.2) node[above left] {$\alpha_3$};
   \draw (2.5,0) node[below left] {$\alpha_1(1)\cdot g_1$};
   \draw (4.5,0) node[below left] {$\alpha_2(1)\cdot g_2$};
   \draw [fill=black] (0,0) circle (1.5pt);
   \draw [fill=black] (3,2) circle (1.5pt);
   \draw [fill=black] (2,0) circle (1.5pt);
   \draw [fill=black] (5,2) circle (1.5pt);
   \draw [fill=black] (4,0) circle (1.5pt);
   \draw [fill=black] (7,2) circle (1.5pt);
   \end{scope}
   \end{tikzpicture}
    \caption{An element of $P_3^G(X)$}
    \label{fig1}
\end{figure}
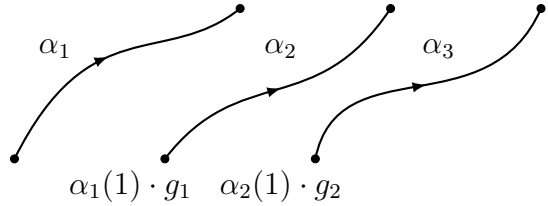

\begin{lemma}\label{lemtwisted}
The $G$-twisted evaluation map $\varepsilon: PX \times G \rightarrow X \times X$ given by $\varepsilon (\alpha, g)=(\alpha(0), \alpha(1)\cdot g)$ is a fibration.
\end{lemma}
\begin{proof}
Note that $\varepsilon$ is the standard fibrational substitute of the ``$G$-saturated diagonal'' $\Delta_G: X \times G \rightarrow X \times X$ given by $\Delta_G(x,g)= (x, xg)$.
\end{proof}

\begin{proposition}\label{propekfib}
For $k \geq 2$, the evaluation map $e_k:P_k^G(X) \rightarrow X \times X$ given by $e_k(\alpha_1, g_1, \ldots, g_{k-1}, \alpha_k)=(\alpha_1(0), \alpha_k(1))$ is a fibration.
\end{proposition}
\begin{proof}
This follows from the commutative diagram with pullback square 
$$\xymatrix{P_k^G(X)\ar@{^(->}[r]\ar[d] \ar@/_{20mm}/[dd]_{e_k} & (PX\times G)^{k-1}\times PX\ar[d]^{\varepsilon^{k-1}\times e_{01}} \\ X\times X^{k-1}\times X \ar@{^(->}[r]_{1\times\Delta^{k-1}\times 1} \ar[d]^{\pi_{1,k+1}} & (X\times X)^k \\ X\times X}
$$
where $\pi_{1,k+1}$ projects onto the first and last coordinates.
\end{proof}

\begin{definition}\label{ladefetc}{\em
The $k$-effective topological complexity ($k$-etc) of the $G$-space $X$,
denoted by $\TC_{\effv}^{G,k}(X)$, is the sectional category of 
$e_k$. 
}\end{definition}

Błaszczyk-Kaluba's $k$-th effective topological complexity $\TC^{G,k}(X)$ is defined in~\cite{BK} as the sectional category of the fibration $\varepsilon_k:P^{G,k}(X) \rightarrow X \times X$. Here 
$$
P^{G, k}(X)=\{(\alpha_1, \ldots, \alpha_k) \in (PX)^k \colon \alpha_i(1) \cdot G = \alpha_{i+1}(0) \cdot G\,\,\, \text{for $1\leq i<k$} \}
$$
and $\varepsilon_k(\alpha_1, \ldots, \alpha_k)= (\alpha_1(0), \alpha_k(1))$. As detailed in the next paragraphs, $k$-etc has a slightly better behavior than Błaszczyk-Kaluba's. For starters, $\TC^{G,k}_{\effv}(X)$ keeps the basic properties of $\TC^{G,k}(X)$ (Propositions~\ref{htpyinvariance} and~\ref{groupnat} below). In fact, the equality $\TC^{G,k}_{\effv}(X)=\TC^{G,k}(X)$ holds for reasonably nice $G$-spaces (Remark~\ref{agreement} below). But more importantly, $k$-etc has better conceptual properties than Błaszczyk-Kaluba's (see Propositions~\ref{estable}, \ref{caracterizaciones} and~\ref{bothsquares} below). 

\begin{remark}\label{agreement}{\em
An element in $P_k^G(X)$ is designed to encode precise ``leaping'' information that assembles a broken path in $P^{G,k}(X)$. In particular, the projection $P\colon P_k^G(X)\to P^{G,k}(X)$ that forgets all the ``group coordinates'' satisfies $e_k=\varepsilon_k\circ P$, so that
\begin{equation}\label{comparacion}
\TC^{G,k}(X)\leq\TC_{\effv}^{G,k}(X).
\end{equation}
This inequality is in fact an equality if the $G$-action on $X$ is principal, for in such a case $P\colon P_k^G(X) \rightarrow P^{G,k}(X)$ is a homeomorphism. (Recall that a free action is principal if the map $\tau: \Im(\Delta_G)\rightarrow G$ satisfying $x \cdot \tau (x, y)=y$ is continuous, where $\Delta_G$ is the $G$-saturated diagonal in the proof of Lemma~\ref{lemtwisted}.)
}\end{remark}

The restriction of $e_{k+1}$ to $P_k^G(X)$ is $e_k$, which readily gives the monotonic behavior
\begin{equation}\label{monobeha}
\TC_{\effv}^{G, k+1}(X)\leq\TC_{\effv}^{G,k}(X),
\end{equation}
an analogue of $\TC^{G, k+1}(X)\leq\TC^{G,k}(X)$ (\cite[Lemma~3.2(2)]{BK}). While Błaszczyk and Kaluba's show the latter inequality to be an equality for principal actions,~(\ref{monobeha}) is sharp under no special restrictions.

\begin{proposition}\label{estable}
 For $k \geq 2$, $\TC_{\effv}^{G,k}(X)=\TC_{\effv}^{G, k+1}(X).$
 \end{proposition}
 \begin{proof}
In view of~(\ref{monobeha}), it suffices to observe that the map $f_k: P_{k+1}^G(X) \rightarrow P_k^G(X)$, $f_k(\alpha_1, g_1, \ldots, \alpha_{k-1}, g_{k-1}, \alpha_k, g_k, \alpha_{k+1})=(\alpha_1, g_1, \ldots, \alpha_{k-1}, g_{k-1}g_k, (\alpha_k\cdot g_k) \star \alpha_{k+1})$, yields a commutative diagram
$$
\xymatrix{P_{k+1}^G(X) \ar[r]^{f_{k}} \ar[d]_{e_{k+1}} & P_k^G(X) \ar[dl]^{e_k} \\ X \times X.}
$$
Here $\star$ stands for concatenation of paths.
\end{proof}

It therefore suffices to restrict attention to $\TC_{\effv}^{G,2}(X)$ which, from now on, will simply be denoted by $\TC_{\effv}^G(X)$.

\begin{proposition}\label{caracterizaciones}
Let $\varepsilon$ and $\Delta_{G}$ be the $G$-twisted evaluation map and the $G$-saturated diagonal, respectively (see Lemma~\ref{lemtwisted} and its proof). Then
$$\TC_{\effv}^G(X)=\secat(\varepsilon)=\secat(\Delta_G).$$
\end{proposition}

\begin{proof} The map $\phi: P_2^G(X) \rightarrow PX \times G$ given by $\phi(\alpha_1, g, \alpha_2)= (\alpha_1 \star (\alpha_2\cdot g^{-1}), g)$ is a homeomorphism with inverse given by $\phi^{-1}(\alpha, g)=(\alpha', g, \alpha'' \cdot g)$, where $\alpha',\alpha'' \in PX$ are given by $\alpha'(t)= \alpha (t/2)$ and $\alpha''(t)= \alpha(\frac{1+t}{2})$ (see Figure~\ref{fig2}).
 \begin{figure}[h]
    \centering
   \begin{tikzpicture}
   \begin{scope}[very thick,decoration={
    markings,
    mark=at position 0.5 with {\arrow{latex}}}] 
   \draw[postaction={decorate},thick] (0, 0) .. controls(1,0.3) .. (2, 1);
   \draw[postaction={decorate},thick] (2, 1) .. controls(3,1.8) .. (4, 2); 
   \draw (0,0) node[below left] {$\alpha(0)$};
   \draw (2,1.2) node[below right] {$\alpha(1/2)$};
   \draw (4.1,2.4) node[below right] {$\alpha(1)$};
   \draw (1,0.4) node[above left] {$\alpha'$};
   \draw (3.1,1.7) node[above left] {$\alpha''$};
   \draw [fill=black] (0,0) circle (1.5pt);
   \draw [fill=black] (2,1) circle (1.5pt);
   \draw [fill=black] (4,2) circle (1.5pt);
   \end{scope}
   \end{tikzpicture}
    \label{fig2}
    \caption{$\alpha=\alpha'\star\alpha''$}
\end{figure}
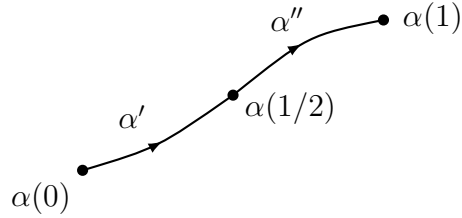
The first asserted equality then follows from the commutative triangle
$$
\xymatrix{P_2^G(X) \ar[r]^{\phi} \ar[d]_{e_2} & PX \times G \ar[dl]^{\varepsilon} \\ X \times X.}
$$
The second asserted equality comes from the commutative diagram
 $$\xymatrix{X \times G \ar[r]^{c \times 1} \ar[d]_{\Delta_G} & PX \times G \ar[dl]^{\varepsilon} \\ X \times X,}$$
where $c: X \xrightarrow{\simeq}PX$ stands for the homotopy equivalence sending $x \in X$ into the stationary path $\overline{x}$.
\end{proof}

Proposition~\ref{htpyinvariance} below is verified along the lines of~\cite[Theorem~3]{MR1957228} (cf.~\cite[Theorem~3.3]{BK}), while Proposition~\ref{groupnat} below is a mild generalization of~\cite[Lemma~3.2(1)]{BK}. Both properties are direct consequences of Definition~\ref{ladefetc}, so we leave the easy details as an exercise for the reader.

\begin{proposition}\label{htpyinvariance}
Let $f\colon X\to Y$ be a $G$-map with a (not necessarily equivariant) right homotopy inverse $g\colon Y\to X$, i.e., $f\circ g\simeq 1_Y$. Then $\TC^G_{\effv}(Y)\leq\TC^G_{\effv}(X)$. In particular $\TC^G_{\effv}$ is a $G$-homotopy invariant.
\end{proposition}

\begin{proposition}\label{groupnat}
For a $G$-trivial space $X$, $\TC_{\effv}^{G}(X)$ agrees with $\TC(X)$, Farber's topological complexity of the underlying space $X$. Further, for a group morphism $\kappa\colon G_1\to G_2$ and a $G_2$-space $X$, we have $\TC_{\effv}^{G_2}(X)\leq\TC_{\effv}^{G_1}(X)$, where $X$ is seen as a $G_1$-space via $\kappa$. In particular, $\TC_{\effv}^{G}(X)\leq\TC(X)$.
\end{proposition}

Proposition~\ref{groupnat} and~(\ref{comparacion}) yield the inequalities $\TC^{G,k}(X)\leq\TC^{G}_{\effv}(X)\leq\TC(X)$ for any $k\geq2$ and general $G$-spaces. In particular,~\cite[Proposition~5.3]{BK} yields Corollary~\ref{nsuytsb} below, while~\cite[Proposition~5.9]{BK} and the paragraph following~(\ref{comparacion}) yield Corollary~\ref{hduiwydg} below. Other calculations in~\cite[Subsection~5.2]{BK} also give full information (gathered in Corollary~\ref{hhssyywoodd} below) in the $\TC_{\effv}^G$ realm.

\begin{corollary}\label{nsuytsb}
Let $p$ be a prime integer. Assume that $\mathbb{Z}_p$ acts cellularly on a positive dimensional sphere $S^n$. If $p=2$, assume in addition that the action preserves orientation. Then $$\TC^{\mathbb{Z}_p}_\effv(S^n)=\begin{cases}1, & \text{if $n$ is odd}; \\ 2, & \text{if $n$ is even}.\end{cases}$$
\end{corollary}

\begin{corollary}\label{hduiwydg}
Let $\mathbb{Z}_2$ be a cellular free action on a positive dimensional sphere $S^n$. Then $\TC^{\mathbb{Z}_2}_\effv(S^n)=1$.
\end{corollary}

\begin{corollary}\label{hhssyywoodd}
Let $\mathbb{Z}_2$ act cellularly on $S^n$. Assume the action is orientation-reversing and that it has an $r$-dimensional fixed point set with $0\leq r\leq n-2$. Then $\TC^{\mathbb{Z}_2}_\effv(S^n)=1$ provided either $n$ is odd, or $n$ is even and the action is linear.
\end{corollary}
\begin{proof}
Proposition~\ref{groupnat} and~\cite[Lemma~5.5]{BK} give $1\leq\TC^{\mathbb{Z}_2,2}(S^n)\leq\TC^{\mathbb{Z}_2}_\effv(S^n)$. The equality $1=\TC^{\mathbb{Z}_2,2}(S^n)$ is obtained in~\cite[Proposition~5.6]{BK} by constructing an explicit 2-ruled motion planner which is effective in their sense. Examination of that construction reveals that the motion planner is actually effective in our sense. The result follows.
\end{proof}

\section{Effectual topological complexity}\label{seceffel}
As explained in~\cite{BK}, the goal in the effective-$\TC$ viewpoint is to motion-plan an autonomous system by taking advantage of potential symmetries. Explicitly, we aim at instructing the system to move from any given initial state to any desired final state, contenting ourselves to arrive at a state that is only $G$-symmetric to the intended destination state. The purpose of this section is to show that the latter task has a completely different nature than the task of motion planning from any given initial state to a desired final \emph{$G$-orbit}. The resulting concept, which we call the \emph{effectual} topological complexity of the system (with respect to the given symmetries), turns out to be a connecting link in the relationship between the effective topological complexity of a $G$-space $X$ and the usual topological complexity of the orbit space $X/G$. As we will see, such a relationship is rather subtle even for an orientable surface endowed with its standard antipodal involution.

\begin{definition}{\em
The effectual topological complexity (ETC) of a $G$-space $X$, denoted by $\TC_{\effl}^G(X)$, is the sectional category of the map $\epsilon: PX \rightarrow X \times (X/G)$ given by the composite $$PX \xrightarrow{e_{01}} X \times X \xrightarrow{1 \times \pi} X \times (X/G),$$ where $\pi$ stands for the canonical projection, and $e_{01}$ is the end-points evaluation map. 
}\end{definition}

 \begin{remark}\label{ANR}{\em
 As shown in~\cite[Lemma~4.1 and Corollary~4.2]{Pavesic}, $\epsilon$ is a fibration if and only if $\pi$ is so, in which case $\TC_{\effl}^G(X)$ is nothing but Pave\v{s}i\'c's  topological complexity of~$\pi$. It will also be convenient to record that, when $\pi$ is a fibration and both $X$ and $X/G$ are compact metric ANR's, the openess requirement for coverings of $X \times (X/G)$ can be waved from the definition of $\TC_{\effl}^G(X)$ (cf.~\cite[Theorem 4.6]{Pavesic}). In such a case, an effectual motion planner for $X$ is a partition of $X\times(X/G)$ by subsets $D_i$ (called the \emph{effectual domains}) together with a family of continuous sections $s_i\colon D_i\to PX$ (called the \emph{effectual instructions}) for the restricted fibrations $\epsilon_i\colon \epsilon^{-1}(D_i)\to D_i$. The effectual motion planner is said to be \emph{optimal} if it has $\TC_{\effl}^G(X)+1$ effectual domains. 
}\end{remark}

In this paper we are mainly interested in free $G$-spaces $X$ for which $\pi: X \rightarrow X/G$ is a covering projection. So, throughout the rest of the paper we assume that $X$ is Hausdorff and that the action of $G$ on $X$ is properly discontinuous (and thus principal, see~\cite[Lemma~14.1.1]{MR2456045}). In particular, $G$ is assumed to be discrete. In such cases,~\cite{Pavesic} gives a thorough study of the basic homotopy properties of ETC. We thus focus on its connections to other TC-invariants:

 \begin{proposition}\label{bothsquares}
 Both squares in the commutative diagram
 \begin{equation}\label{doscuadros}
 \xymatrix{P_2^G(X) \ar[r]^{\;\;q} \ar[d]_{e_2} & PX \ar[r]^{P\pi\;} \ar[d]^{\epsilon} & P(X/G) \ar[d]^{e_{01}} \\ X \times X \ar[r]_{1\times \pi\;\;\;\;\;\;\;} & X \times (X/G) \ar[r]_{\pi \times 1\;\;\;\;\;\;} & (X/G) \times (X/G)}
 \end{equation}
 are strict pullbacks. Here $q(\alpha, g, \beta)=\alpha \star(\beta g^{-1})$. In particular (\ref{sandwich}) holds.
 \end{proposition}

 \begin{proof}
Consider the commutative diagram  
  \begin{center}
  \begin{tikzcd}
    PX \times G\ar[r,"\phi^{-1}"]\ar[rd,"\varepsilon"'] & P_2^GX\ar[r,"q"]\ar[d,"e_2"] & PX\ar[d,"\epsilon"] &
  \\ 
   & X\times X\ar[r,"1\times \pi"'] & X\times (X/G) & 
  \end{tikzcd}
  \end{center}
where $\phi$ is the homeomorphism in the proof of Proposition~\ref{caracterizaciones}, so that the top horizontal composite is projection onto the the first coordinate. Let $P$ be the pullback of $1 \times \pi$ and $\epsilon$. Note that the canonical map $\varphi: PX \times G \rightarrow P$ is surjective: given $(x,y, \gamma) \in P$, so that $x=\gamma(0)$ and $[y]=[\gamma(1)]$, say $y=\gamma(1)\cdot g$ with $g \in G$, we have $\varphi(\gamma, g)=(x,y, \gamma)$. Since $X$ is $G$-free, $\varphi$ is injective. Thus the assertion for the left hand-side square in~(\ref{doscuadros}) will follow once we show that $\varphi: PX \times G \rightarrow P$ is a topological embedding. 

Since $G$ is discrete, $PX\times G = \bigsqcup_{g \in G} PX \times \{g\}$ has the disjoint-union topology. On the other hand, for each $g \in G$, the commutative diagram 
$$  \begin{tikzcd}
    PX\times\{g\}\ar[r,hook] \ar[rrrd,equal] & PX\times G\ar[r,"\varphi"] & P\ar[r,hook] & X\times X \times PX\ar[d,"\pi_3"] \\ & & & PX,
  \end{tikzcd}
$$
where $\pi_3$ is the projection onto the third coordinate, shows that the restriction of $\varphi$ to $PX \times \{g\}$ is a topological embedding. Consequently, it is enough to check that $P$ has the disjoint-union topology $\bigsqcup_{g \in G} I_g$, where $I_g$ is the image of $PX \times \{g\}$ under~$\varphi$. We argue in fact that, in $X \times X \times PX$, $\overline{I_{g_1}} \cap I_{g_2} = \varnothing$ for $g_1 \neq g_2$: Assume there is an element $(\gamma(0), \gamma(1)\cdot g_2, \gamma) \in \overline{I_{g_1}} \cap I_{g_2}$, and take a neighborhood $W$ of $\gamma(1)$ such that $W{\cdot\,} g \,\cap\, W = \varnothing$ whenever $g \neq e$. Since $X \times (W{\cdot\,} g_2) \times [\{1\}, W]$ is a neighborhood of $(\gamma(0), \gamma(1)\cdot g_2, \gamma)$ in $X\times X\times PX$, there exists and element $(\alpha(0), \alpha(1)\cdot g_1, \alpha) \in (X \times (W{\cdot\,}g_2) \times [\{1\}, W]) \cap I_{g_1}$. Therefore $\alpha(1)\cdot g_1 \in W {\cdot\,} g_2$ and $\alpha(1) \in W$, which yields $\alpha(1)\cdot g_1 \in W{\cdot\,}g_1 \cap W{\cdot\,}g_2$, and so $g_1=g_2$.
 
We now deal with the right hand-side square in~(\ref{doscuadros}), namely, the exterior commutative square in the diagram
$$
  \begin{tikzcd}
    PX\ar[rd,"\psi"]\ar[rrd,bend left,"P\pi"']\ar[rdd,bend right,"\epsilon"'] & &  \\
     & B\ar[r]\ar[d] & P(X/G)\ar[d,"e_{01}"] \\
   & X\times (X/G)\ar[r,"\pi\times 1"'] & (X/G)\times (X/G),
  \end{tikzcd}
$$
where $B$ is the pullback of $\pi \times 1$ and $e_{01}$. The canonical map $\psi:PX \rightarrow B$ is one-to-one because $\pi$ is a covering map, so we only need to check continuity of $\psi^{-1}$. Consider the commutative diagram with pullback square
  \begin{center}
  \begin{tikzcd}
    PX\ar[rd,]\ar[rrd,bend left,"P\pi"]\ar[rdd,bend right,"e_0"'] & & \\
     & Q\ar[r]\ar[lu,bend left,"\ell"]\ar[d] & P(X/G)\ar[d,"e_0"] \\
     & X\ar[r,"\pi"'] & X/G 
    \end{tikzcd}
  \end{center}
 where the lifting $\ell$ for the canonical map $PX\to Q$ is continuous in view of~\cite[Theorem~II.7.8]{MR0210112} (recall $\pi$ has been assumed to be a covering projection). In these conditions, $\psi^{-1}$ is continuous as it factors as the composite of the two bottom horizontal maps in the commutative diagram
  \begin{center}
  \begin{tikzcd}
    X\times X/G\times P(X/G)\ar[r,"\pi_{13}"] & X\times P(X/G) & \\
     B\ar[r]\ar[u,hook] & Q\ar[r,"\ell"]\ar[u,hook] & PX,
  \end{tikzcd}
  \end{center}
 where $\pi_{13}$ projects onto the the first and third coordinates.  
 \end{proof}

 \begin{remark}{\em
Pave\v{s}i\'c shows in \cite{Pavesic} that the right hand-side square in~(\ref{doscuadros}) is a homotopic pullback whenever $\pi : X \rightarrow X/G$ is a fibration.
}\end{remark}


\section{Torus}\label{sectoro}
We now deal with the middle inequality in~(\ref{cfto}), i.e.:

\begin{theorem}\label{T}
$\TC_{\effl}^{\Z_2}(\T)=3$.
\end{theorem}

Before proving this fact, we discuss a couple of important consequences.
First of all, at the end of~\cite{Pavesic}, Pave\v{s}i\'c suggests the posibility that 
$\TC_{\effl}^{G}(X)=\TC(X)$ for every finitely sheeted covering projection $\pi: X \rightarrow X / G$. Theorem~\ref{T} gives an explicit counterexample to such a situation. Secondly, recall that the calculation of the topological complexity of the Klein bottle $\K$ in \cite{CV} amounts to showing the non-vanishing of the mod-2 reduction of $\nu_{\K}^{\TC}$, the primary (and unique) obstruction responsible for the equality $\TC(\K)=4$. In view of the functoriality of primary obstructions (\cite[Theorem~VI.6.3]{MR516508}), Theorem~\ref{T} and the pullback diagram 
\begin{center}    
    \begin{tikzpicture}[commutative diagrams/every diagram]
    \matrix[matrix of math nodes, name=m, commutative diagrams/every cell] 
    {P\T & P\K\\
    \T \times \K & \K \times \K\\};
    \path[commutative diagrams/.cd, ->, every label, >=latex]
    (m-1-1) edge node [above]{$P\pi$}(m-1-2)
    (m-1-1) edge node [left]{$\epsilon$}(m-2-1)
    (m-1-2) edge node [right]{$e_{01}$} (m-2-2)
    (m-2-1) edge node[below] {$\pi \times 1$} (m-2-2);
    \end{tikzpicture}
\end{center}
show that $\nu_{\K}^{\TC}$ maps trivially under the map $\pi \times 1\colon \T \times \K \longrightarrow \K \times \K$. These observations lead to the following retrospective explanation of Cohen-Vandembroucq's successful mod-2 calculations.

\begin{corollary}\label{sccssfulcv}
Cohen-Vandembroucq's non-trivial obstruction $\nu_{\K}^{\TC}$ is 2-torsion.
\end{corollary}
\begin{proof}
We recall a few preliminary facts from~\cite{CF}. Let $I$ stand for the augmentation ideal of the fundamental group $\pi_1(\K)$, i.e., $I$ is the kernel of the augmentation map $\Z\left[\pi_{1}(\K)\right] \to \Z$. The fundamental group $\pi_{1}(\K \times \K)=\pi_{1}(\K) \times \pi_{1}(\K)$ acts on $\Z[\pi_1(\K)]$ and, by restriction, on $I$ via the formula
$$(a, b) \cdot \Sigma n_{i} c_{i}=\sum n_{i}\left(a c_{i} b^{-1}\right).$$ 
Lastly, the obstruction $\nu_{\K}^{\TC}$ lies in the twisted cohomology group $H^{4}\left(\K \times \K; I^{\otimes 4}\right)$.
Cohen and Vandembrouq assess $\nu_{\K}^{\TC}$ through its Poincar\'e-dual image
\begin{equation*}
    \begin{array}{rccl}
    & H^{4}\left(\K \times \K; I^{\otimes 4}\right) & \longrightarrow & H_0\left(\K \times \K; I^{\otimes 4} \otimes \;\widetilde{\Z}\right) \\
    & \nu_{\K}^{\TC} & \mapsto & \nu_{\K}^{\TC} \cap [\K \times \K],
\end{array}  
\end{equation*}
where $\widetilde{\Z}$ stands for the orientation module of $\K \times \K$, and $[\K \times \K]$ is the corresponding twisted fundamental class---a generator of $H_{4}(\K \times \K; \widetilde{\Z}) \cong \mathbb{Z}$. A similar obstruction-theory setting holds when the fibration $e_{01}: P \K \rightarrow \K \times \K$ is replaced by its pullback fibration $\epsilon: P \T \longrightarrow  \T\times \K$. In particular, we highlight that the orientation module $\overline{\Z}$ of $\T \times \K$ 
is the pull back of $\widetilde{\Z}$ under $\pi \times 1: \T \times \K \rightarrow \K \times \K$.

As a last preliminary ingredient in the proof, consider the group presentations
$$\pi_{1}(\T)=\left\langle a, b: a b=b a\right\rangle
\quad \text{ and } \quad
\pi_{1}(\K)=\left\langle x, y: y x y=x\right\rangle,$$
with generators chosen so that the covering
$\pi: \T \to \K$ has $\pi_{*}(a)=x^{2}$ and $\pi_{*}(b)=y$ (see Figure~\ref{2c}).
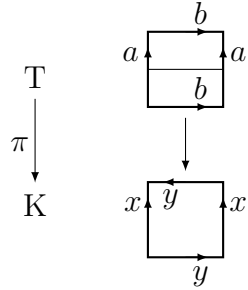
\begin{figure}[h]
\begin{center}
\begin{tikzpicture}[scale=1]  
\draw[thick] (0,0) rectangle (1,1);
\draw[thick] (0,-2) rectangle (1,-1);
\node (a) at (-0.23, 0.7) {$a$};
\node (a) at (1.2, 0.7) {$a$};
\node (b) at (0.7, 0.25) {$b$};
\node (b) at (0.7, 1.25) {$b$};
\node (x) at (-0.2, -1.3) {$x$};
\node (x) at (1.2, -1.3) {$x$};
\node (y) at (0.3, -1.23) {$y$};
\node (y) at (0.7, -2.25) {$y$};
\node (A) at (0.5, 0) {};
\node (B) at (0.5, -1) {};
\draw[->, >=latex] (A) edge (B);
\node (C) at (-1.5, .4) {$\T$};
\node (D) at (-1.5, -1.3) {$\K$};
\draw[->, >=latex] (C) edge (D);
\node (F) at (-1.7, -0.5) {$\pi$};
\draw[->, >=latex] (0.5,0) -- (0.8,0);
\draw[->, >=latex] (0.5,1) -- (0.8, 1);
\draw[->, >=latex] (0.5,-1) -- (0.2,-1);
\draw[->, >=latex] (0.5,-2) -- (0.8, -2);
\draw[->, >=latex] (0, 0.5) -- (0, 0.8);
\draw[->, >=latex] (1, 0.5) -- (1, 0.8);
\draw[->, >=latex] (0, -1.5) -- (0, -1.2);
\draw[->, >=latex] (1, -1.5) -- (1, -1.2);
\draw[-, >=latex, thin] (0, .5) -- (1, .5);
\end{tikzpicture}
\end{center}
    \caption{The double cover of the Klein bottle by the torus}
    \label{2c}
\end{figure}
It is standard that the corresponding elements $x_{1}, x_{2} \in \pi_{1}(\K \times \K)$ act on $\widetilde{\Z}$ by interchanging sign, while 
$y_{1}, y_{2} \in \pi_{1}(\K \times \K)$ act trivially. Subindices are used to indicate coordinate source. Likewise, 
$a_{1}, b_{1}, y_{2} \in \pi_1(\T \times \K)$ act trivially on $\overline{\Z}$, while $x_{2} \in \pi_{1}\left(\T \times \K \right)$ acts by sign interchange. With this preparation, a direct group-cohomology calculation (left as an exercise for the reader) shows that the induced map 
$$\left(\pi \times 1\right)_{*}: H_{4}\left(\T \times \K ; \overline{\Z}\right) \cong \Z \longrightarrow \Z \cong H_{4}(\K \times \K ; \widetilde{\Z})$$
is multiplication by $-2$, i.e., $\left(\pi \times 1\right)_{*}$ takes the twisted fundamental class $[\T \times \K]$ onto $-2[\K \times \K]$. Thus
\begin{align*}
\left(2\hspace{.5mm}\nu_{\K}^{\TC}\right) \cap[\K \times \K]
&=\nu_{\K}^{\TC} \cap (2[\K \times \K]) \\
&=\nu_{\K}^{\TC} \cap\left(\rule{0mm}{3.4mm}-(\pi \times 1)_{*}
[\T \times \K]\right) \\
&=-(\pi \times 1)_{*}\left((\pi \times 1)^{*}\left(\nu_{\K}^{\TC}\right) \cap\left[\T \times \K\right]\right),
\end{align*}
where the latter expression vanishes, since $(\pi \times 1)^{*}\left(\nu_{\K}^{\TC}\right)=0$, as observed at the end of the paragraph following the statement of Theorem~\ref{T}. Poincar\'e duality then yields $2\nu_{\K}^{\TC}=0$.
\end{proof}

The rest of the section is devoted to the proof of Theorem \ref{T}.

\smallskip
The following cohomology facts are standard and easy to prove (all cohomology groups below are taken with mod-2 coefficients). $H^*(\T)$ is generated by elements $\alpha, \beta \in H^{1}(\T)=\Z_{2} \oplus \Z_{2}$ subject to the relations $\alpha^{2}=\beta^{2}=0$ and $\alpha \beta=\gamma$, where $\gamma$ stands for the generator of $H^{2}(\T)=\Z_{2}$. Likewise, $H^{*}(\K)$ is generated by elements $\kappa, \lambda \in H^{1}(\K)=\mathbb{Z}_{2} \oplus \Z_{2},$ subject to the relation $\kappa \lambda=0$ and $\kappa^{2}=\lambda^{2}=\mu,$ where $\mu$ stands for the generator of $H^{2}(\K)=\Z_{2}$. Furthermore, the map $\pi^{*}: H^{*}(\K) \rightarrow H^{*}(\T)$ is determined by $\pi^{*}(\kappa)=\pi^{*}(\lambda)=\alpha+\beta$. With this information, the inequality $\TC_{\effl}^{\Z2}(\T) \geq 3$ follows easily from the well known estimate that the sectional category of a fibration $p\colon E \to B$ is bounded from below by the cup length of elements in the kernel of $p^{*}: H^{*}(B) \to H^{*}(E)$. Explicitly, in our situation, the fibration under consideration is $\epsilon\colon P\T \rightarrow \T \times \K$ which, in terms of the standard homotopy equivalence $\T \simeq P\T$, takes the form $(1, \pi): \T \to \T \times \K$. Then $\left(1, \pi)^{*}(\alpha \otimes 1+\beta \otimes 1 +1 \otimes \lambda\right)=\alpha+\beta+(\alpha+\beta)=0,$ whereas a direct calculation yields $\left(\alpha \otimes 1+\beta \otimes 1 +1 \otimes \lambda\right)^{3}=(\alpha+\beta) \otimes \mu \neq 0$.

\smallskip
More difficult is establish $\TC_{\effl}^{\Z_2}(\T) \leq 3$. Rather than using obstruction theory, we actually describe an (optimal) effectual motion planner with 4 domains, i.e., a partition of $\T \times \K$ into 4 subsets (not necessarily open, in view of Remark~\ref{ANR}), each admitting a
section for the corresponding restriction of $\varepsilon: P \T \to \T \times \K$.

\smallskip
We start by fixing some notation. Think of $\T$ as $\T=S^1 \times S^1$, where the first (second) $S^1$-coordinate will be depicted horizontally (vertically). In these terms, the antipodal involution on $\T$ becomes $\sigma(x_1,x_2)=(-x_1,\overline{x_2})$, where $\overline{z}$ stands for the complex conjugate of $z\in S^1$. For a point $x=\left(x_{1}, x_{2}\right) \in \T$, set $V_{x}:=\left\{x_{1}\right\} \times S^{1}$ ($H_{x}:=S^{1}\times\left\{x_{2}\right\}$), the ``vertical'' (``horizontal'') circle passing through $x$.
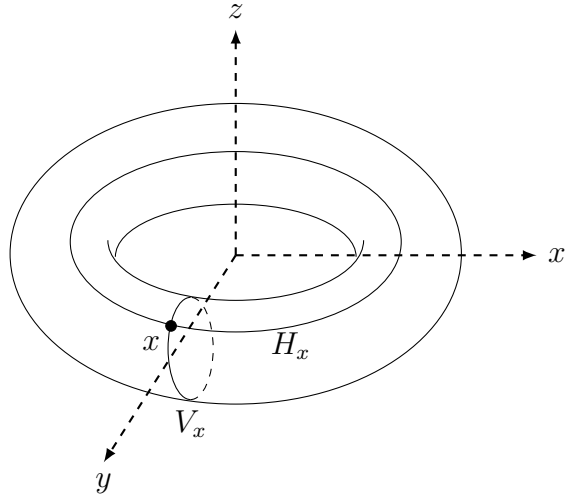
$$
\begin{tikzpicture}[scale=2]  
\draw (4.8,-0.01) arc (0:180:0.8 and 0.35); 
\draw (3.15,0.1) arc (180:360:0.85 and 0.4) node[above left = 0.5cm] {}; 
\draw (5.5, 0.01) arc (0:360:1.5 and 1) node[above right] {}; 
\draw (3.7,-0.28) arc (90:270:0.15 and 0.34) node[below] {$V_x$};
\draw[dashed] (3.7,-0.96) arc (-90:90:0.15 and 0.34);
\draw (5.1, 0.09) arc (0:360:1.1 and 0.6) node[below left = 1.5cm] {$H_x$}; 
\filldraw[black](3.57,-.47) circle (1pt) node[below left]{$x$};
\draw[thick, ->, >=latex, dashed] (4,0,0) -- (6,0,0) node[right]{$x$};
\draw[thick, ->, >=latex, dashed] (4,0,0) -- (3.7,-0.8,1.5) node[below]{$y$};
\draw[thick, ->, >=latex, dashed] (4,0,0) -- (4,1.5,0) node[above]{$z$};
\end{tikzpicture}
$$
Set in addition:
\begin{itemize}
\item $M_{x}:=\left\{e^{i \theta} x_{1}\colon-\frac{\pi}{2} \leq \theta \leq \frac{\pi}{2}\right\} \times S^{1}$, the half handle determined by $x$.\vspace{-3mm}

\item $C_{x}^{I}:=\left\{e^{-\frac{i \pi}{2}}x_1\right\} \times S^{1}$, the ``left'' boundary component of $M_x$.\vspace{-3mm}

\item $C_{x}^{D}:=\left\{e^{\frac{i \pi}{2}}x_1\right\} \times S^{1}$, the ``right'' boundary component of $M_x$.\vspace{-3mm}

\item $C_{x}:=\left(S^{1} \times\left\{-x_{2}\right\}\right) \cap M_{x}$.\vspace{-3mm}

\item $A_{x}:=C_{x}^{I} \, \cup \, C_{x} \, \cup \, C_{x}^{D}$.\vspace{-3mm}

\item $x':=\left(x_{1},-x_{2}\right),$ the ``antipodal to $x$ in $V_{x}$''.\vspace{-3mm}

\item $a_{x}:=\left(e^{\frac{i \pi}{2}} x_{1},-\overline{x_2}\right)$, $b_{x}:=\left(e^{-\frac{i \pi}{2}} x_{1},-x_{2}\right), \text{ so \,} C_{x}^{I} \,\cap\, C_{x}=\left\{b_{x}\right\} \text{ and \;}\sigma(a_x)=b_x$.
\end{itemize}
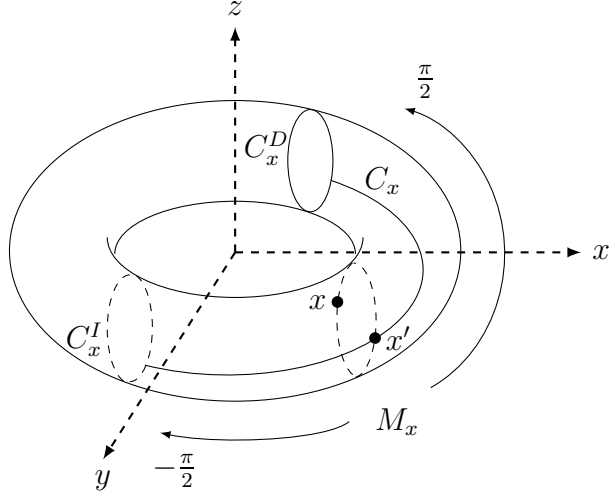
$$\begin{tikzpicture}[scale=2]  
\draw (4.8,-0.01) arc (0:180:0.8 and 0.35); 
\draw (3.15,0.1) arc (180:360:0.85 and 0.4); 
\draw (5.5,0.01) arc (0:360:1.5 and 1); 
\draw (4.5,0.95) arc (90:270:0.15 and 0.34);
\draw (4.5,0.27) arc (-90:90:0.15 and 0.34) node[below left=0.2cm] {$C_x^D$};
\draw[dashed] (3.3,-0.15) arc (90:270:0.15 and 0.36) node[above left=0.3cm] {$C_x^I$};
\draw[dashed] (3.3,-0.86) arc (-90:90:0.15 and 0.36);
\draw (3.4,-0.75) arc (-115:58:1.3 and 0.7) node[right=3mm] {$C_x$};
\draw[dashed] (4.78,-0.07) arc (108:270:0.15 and 0.39);
\draw[dashed] (4.8,-0.83) arc (-85:90:0.15 and 0.38);
\filldraw[black](4.68,-0.33) circle (1pt) node[left]{$x$};
\filldraw[black](4.93,-0.57) circle (1pt) node[right]{$x'$};
\draw[<-, >=latex] (3.5,-1.2) arc (-135:-14:0.75 and 0.16) node [right=0.2cm] {$M_x$};
\draw[->, >=latex] (5.3,-.9) arc (-63:75:.9 and 1)  node [above right=0mm] {$\frac{\pi}{2}$}; 
\node at (3.6, -1.45) {$-\frac{\pi}{2}$};
\draw[thick, ->, >=latex, dashed] (4,0,0) -- (6.3,0,0) node[right]{$x$};
\draw[thick, ->, >=latex, dashed] (4,0,0) -- (3.7,-0.8,1.5) node[below]{$y$};
\draw[thick, ->, >=latex, dashed] (4,0,0) -- (4,1.5,0) node[above]{$z$};
\end{tikzpicture}
$$
Finally set $\A:=H_{(1,-1)}=S^1 \times\{-1\}$ ($\B:=H_{(1,1)}=S^{1} \times\{1\}$), the ``inner'' (``outer'') horizontal circle. Note that both $\A$ and $\B$
are closed under the involution.

\smallskip
We are ready to define the domains $D_i$ $(i=1,2,3,4)$ and corresponding sections that complete the proof of Theorem~\ref{T}. The first domain is 
$$D_{1}:=\left\{(x, z) \in \T \times \K\colon \text{there exists \,} y \in M_{x}\setminus A_{x} \text { with } z=\pi(y)\right\}.$$
Note that $\sigma(y)\notin M_{x}$ if $y \in M_{x}\setminus A_{x}$. Therefore the $\pi$-preimage $y$ of $z$ in the definition of $D_1$ is unique. Furthermore, such an element $y\in\pi^{-1}(z)\cap(M_x\setminus A_x)$ clearly depends continuously on $(x,z)\in D_1$. Thus, a section $s_1:D_1 \to P\T$ as the one we need sends a pair $(x,z) \in D_1$ into the path in $\T$ from $x$ to $y$ depicted by the thick arrows in Figure~\ref{f1}, i.e., we first adjust the first $S^1$-coordinate, and then adjust the second $S^1$-coordinate. The continuity on $x$ and $y$ of these adjustments comes from the facts that $y\in M_x$ (for the first adjustment) and  $y\notin A_x$ (for the second adjustment). 

\smallskip
Note that $\T \times \K \setminus D_{1}=\left\{(x, z)\colon \pi^{-1}(z) \cap A_x \neq \varnothing \right\}.$
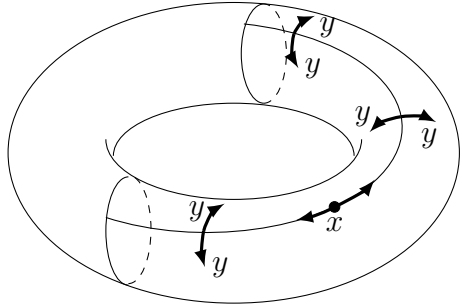
\begin{figure}[h!]
    \centering
\begin{tikzpicture}[scale=2]  
\draw (4.8,-0.01) arc (0:180:0.8 and 0.35); 
\draw (3.15,0.1) arc (180:360:0.85 and 0.4); 
\draw (5.5,0.01) arc (0:360:1.5 and 1); 
\draw (4.2,1) arc (90:270:0.15 and 0.33);
\draw[dashed] (4.2,0.34) arc (-90:90:0.15 and 0.33);
\draw (3.3,-0.15) arc (90:270:0.15 and 0.36);
\draw[dashed] (3.3,-0.86) arc (-90:90:0.15 and 0.36);
\draw (3.15,-0.42) arc (-121:79:1.3 and 0.7); 
\filldraw[black](4.67,-0.35) circle (1pt) node[below]{$x$};
\draw[->,>=latex, very thick] (4.67,-0.35) arc (-20:-45:1.1 and .24);
\draw[->,>=latex, very thick] (4.67,-0.35) arc (-10:-4:21 and 1.75);
\draw[->,>=latex, very thick] (3.8,-.51) arc (165:120:0.3 and 0.3) node[left=1mm]{\raisebox{-4mm}{$y$}};
\draw[->,>=latex, very thick] (3.8,-.51) arc (165:200:0.3 and 0.4) node[right=0.01cm]
{$\!y$};
\draw[->,>=latex, very thick] (5.11,.25) arc (100:60:0.35 and 0.3) node[below right=-0.11cm]{$\!\!\!\!y$};
\draw[->,>=latex, very thick] (5.11,.25) arc (100:140:0.35 and 0.3) node[left=-0.04cm]{\raisebox{5mm}{$y\!\!$}};
\draw[->,>=latex, very thick] (4.4,.8) arc (165:120:0.3 and 0.2) node[right=-1mm]{\raisebox{-5mm}{$y$}};
\draw[->,>=latex, very thick] (4.4,.8) arc (165:210:0.3 and 0.3) node[right]{$\!y$};
\end{tikzpicture}
\caption{Motion planning in $D_1$}
\label{f1}
\end{figure}
The second domain is 
$$D_{2}:=\{(x, z)\colon x \notin \A \cup \B \text{ and there exists }y \in A_{x}\setminus\left(C_{x}^{I} \cup\left\{a_{x}\right\}\right) \text {with } \pi(y)=z\}.$$
As in the case of $D_1$, the $\pi$-preimage $y$ of $z$ in the definition of $D_{2}$ is unique and depends continuously on $(x, z) \in D_{2}$. Since $C^D_x$ is closed, the continuity assertion is not completely obvious when $y \in C_{x}^{D}\setminus\left\{a_{x}\right\}$ and, for such a case, we offer the following argument. Let~$V$ be a small neighborhood of~$y$, say small enough so that $V$ and $\sigma(V)$ are contained in different sides of the hyperplane $P_x$ in $\mathbb{R}^3$ spanned by $V_x$ (see Figure~\ref{tubular}). Since $\sigma(y) \neq b_x \text { (for } y \neq a_x)$, we can shrink~$V$ if needed so to assume that $\sigma(V)$ does not intersect some tubular neighborhood $\tau$ of $C_x$ (see Figure~\ref{tubular}).
\begin{figure}[h!]
\begin{center}
\begin{tikzpicture}[scale=2]  
\draw (4.8,-0.01) arc (0:180:0.8 and 0.35); 
\draw (3.15,0.1) arc (180:360:0.85 and 0.4); 
\draw (5.5,0.01) arc (0:360:1.5 and 1); 
\draw (3.21,-0.05) arc (70:200:0.37 and 0.37) node[below left=-1.5mm] {$C_x^I$};
\draw[dashed] (2.76,-0.55) arc (-130:40:0.32 and 0.35);
\draw (5.37,-0.4) arc (0:115:0.4 and 0.4);
\draw[dashed] (4.8,-0.03) arc (-200:-50:0.34 and 0.35) node[right] {$C_x^D$};
\draw (2.95,-0.055) arc (185:355:1.11 and 0.5);
\filldraw[black](2.95,-0.055) circle (0.5pt) node[above=0.1cm]{$\!\!\!\!\!\!\!\!b_x$};
\draw[dashed] (3.05,-0.04) arc (190:350:1.04 and 0.5);
\draw[dashed] (2.85,-0.08) arc (190:350:1.21 and 0.6) node[below left = 1.2cm] {$\tau$};
\draw[dashed] (3.05,-0.04) arc (0:180:0.1);
\draw[dashed] (5.24,-0.08) arc (0:160:0.07 and 0.12);
\filldraw[black](5.17,-0.058) circle (0.5pt);
\draw[dash dot](3.2,-0.52) circle (1.5pt);
\filldraw[black](3.2,-0.52) circle (0.5pt) node[left]{} node[below right]{$\!\!\!\!\sigma(V)$};
\draw[dash dot](5.31,-0.2) circle (1.5pt);
\filldraw[black](5.31,-0.2) circle (0.5pt) node[below left]{$y\!\!$} node[above right]{$\!\!V$};
\filldraw[black](5.17,-0.49) circle (0.5pt) node[below=-0.03] {$a_x$};
\draw (4.3,-0.49) arc (0:85:0.16 and 0.19);
\draw[dashed] (4.14,-0.3) arc (90:200:0.17 and 0.38);
\filldraw[black](3.985,-0.81) circle (0.5pt) node[right]{$x$};
\filldraw[black](4.3,-0.49) circle (0.5pt) node[below=0.02cm]{$x'$};
\draw[-, >=latex] (3.7,1.8,0) -- (4.5,.8,0);
\draw[-, >=latex] (4.5,-1.7,0) -- (4.5,.8,0);
\draw[-, >=latex, dashed] (4.5,-1.7,0) -- (3.7,-.7,0);
\draw[-, >=latex] (4.5,-1.7,0) -- (3.94,-1,0);
\draw[-, >=latex,dashed] (3.7,-.7,0) -- (3.7,-.23,0);
\draw[-, >=latex] (3.7,-.28,0) -- (3.7,.32,0);
\draw[-, >=latex, dashed] (3.7,.32,0) -- (3.7,1,0);
\draw[-, >=latex] (3.7,1,0) -- (3.7,1.83,0);
\node at (4.7,-1.5) {$P_x$};
\end{tikzpicture}
\end{center}
\caption{Neighborhoods $V$, $\sigma(V)$ and $\tau$}\label{tubular}\end{figure}
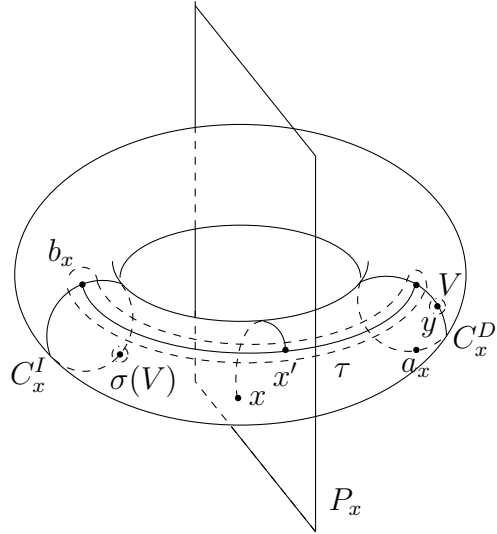
Since $x'$, $C_x$ and $C^D_x$ depend continuously on $x$, there is a neighborhood $U$ of $x$ such that, for any $x_1 \in U$, $C_{x_1} \subseteq \tau$ and $C^D_{x_1}$ lies on the same side of $P_x$ as $C^D_x$ does. Under such conditions, for an element $\left(x_1, z_1\right) \in(U \times \pi(V)) \cap D_{2}$, the element $y_1 \in A_{x_1}\setminus\left(C_{x_1}^I \cup \{a_{x_1}\}\right)$ satisfying $\pi\left(y_1\right)=z_1$ will also satisfy $y_1 \in V \cup \sigma(V)$ and, for the required continuity, we need to make sure that in fact $y_1 \in V$. Assume, for a contradiction, that $y_1\in\sigma(V)$. Given the choosing of $V$, it follows that $y_1\in C_{x_1}$. So $y_1 \in C_{x_1} \subseteq \tau$, which contradicts $\tau \cap \sigma(V) = \varnothing$. Having established the fact that $y\in\left(A_x\setminus(C^I_x\cup a_x)\right)\cap\pi^{-1}(z)$ depends continuously on $(x,z)\in D_2$, the rest is easy: A section $s_2:D_2 \to P\T$ as the one we need sends a pair $(x,z) \in D_2$ into the path in $\T$ from $x$ to $y$ depicted by the thick arrows in Figure~\ref{laded2}. Note that the condition $x\notin\A\cup\B$ implies that $a_x$ is not the intersection of the half horizontal arc $C_x$ with the vertical circle $C_x^D$, so that the motion planning is precisely as indicated in Figure~\ref{laded2}.
\begin{figure}
\begin{center}
\begin{tikzpicture}[scale=2]  
\draw (4.8,-0.01) arc (0:180:0.8 and 0.35); 
\draw (3.15,0.1) arc (180:360:0.85 and 0.4); 
\draw (5.5,0.01) arc (0:360:1.5 and 1); 
\draw (4.3,0.98) arc (90:270:0.15 and 0.33);
\draw (4.3,0.32) arc (-90:90:0.15 and 0.33);
\draw[dashed] (2.51,-0.08) arc (180:0:0.35 and 0.18);
\draw[dashed] (2.51,-0.09) arc (-180:0:0.35 and 0.18);
\draw (4.45,0.7) arc (68:-155:1.1 and 0.7); 
\filldraw[color = black, fill=white] (4.15, 0.7) circle (1pt) node {};
\filldraw[color = black, fill=white] (3.03, -0.23) circle (1pt) node {};
\node (D) at (3.5, 1.2) {$[a_x]$};
\draw[->,>=latex] (3.3,1.15) to[out=210, in=90]  (3.03,-0.15);
\draw[->,>=latex] (3.7,1.15) to[out=-10, in=90]  (4.15,0.8);
\draw[->,>=latex,very thick] (4.38,0.73) arc (40:140:0.1 and 0.5) node[below=-0.2cm] {\scriptsize $y\!\!\!\!$};
\draw[->,>=latex, very thick] (4.35,0.45) arc (-60:-130:0.1 and 0.5) node[above=-1.4mm] {\scriptsize $\,\,\,y$};
\draw[<-,>=latex, very thick] (4.3,-0.635) node [below=-.4mm] {\scriptsize $y$} arc (9:150:0.1 and 0.3);
\draw[->,>=latex, very thick] (4.15,-0.75) arc (-95:-120:2 and 1.2) node[left=-0.1cm] {\scriptsize $y$};
\draw[->,>=latex, very thick] (4.45,-0.5) arc (-85:-31:0.82 and 1.15) node [above left=-0.2cm] {\scriptsize $y$};
\draw[->,>=latex, very thick] (4.99,0.2) arc (5:40:1.7 and 0.66) node [below=0cm] {\scriptsize $y$};
\draw[dashed] (4.35,-0.635) arc (9:160:0.15 and 0.4);
\filldraw[] (4.35,-0.62) circle (0.5pt);
\filldraw[] (4.06,-0.55) circle (0.5pt) node[left=-0.1cm] {$x$};
\end{tikzpicture}
\end{center}
\caption{Motion planning in $D_2$}\label{laded2}\end{figure}
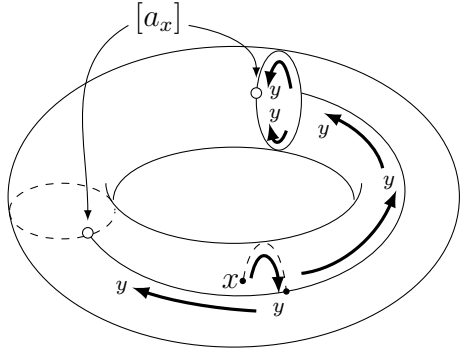

\smallskip
Note that $\T \times \K\setminus\left(D_{1} \cup D_{2}\right)$ consists of the pairs $(x, z)$ satisfying one of the following two conditions:
\begin{itemize}
    \item [(i)] $x \notin \A \cup \B$ and $z=\left[a_{x}\right]$.
    \item [(ii)] $x \in \A \cup \B$ and $\pi^{-1}(z) \cap A_{x} \neq\varnothing$. 
\end{itemize}
The third domain combines the pairs satisfying (i) with some of the pairs satisfying~(ii). Let $D_{31}$ consist of the pairs $(x, z) \in \T \times \K$ satisfying (i) above. Since $\left[a_{x}\right]=\left[b_{x}\right]$, a section $s_{31}: D_{31} \to P\T$
as the one we need sends $(x, [a_x]) \in D_{31}$ into the path in $\T$ from $x$ to $b_x$ depicted by the thick arrows in Figure~\ref{d31}.
\begin{figure}
\begin{center}
\begin{tikzpicture}[scale=2]  
\draw (4.8,-0.01) arc (0:180:0.8 and 0.35); 
\draw (3.15,0.1) arc (180:360:0.85 and 0.4); 
\draw (5.5,0.01) arc (0:360:1.5 and 1); 
\draw (4.2,1) arc (90:270:0.15 and 0.33);
\draw[dashed] (4.2,0.34) arc (-90:90:0.15 and 0.33);
\draw (3.7,-0.28) arc (90:270:0.15 and 0.34);
\draw[dashed] (3.7,-0.96) arc (-90:90:0.15 and 0.34);
\draw (3.55,-0.6) arc (-100:79:1.4 and 0.7); 
\draw[dashed] (4.32,-0.27) arc (90:173:0.15 and 0.45);
\draw (4.35,-0.27) arc (90:28:0.15 and 0.45);
\filldraw[black] (3.55, -0.6) circle (1pt) node [left]{$b_x$};
\filldraw[black] (4.17, -0.72) circle (1pt) node [below]{$x$};
\draw[thick, ->,>=latex, very thick] (4.235,-0.7) arc (160:46:0.12 and 0.45);
\draw[thick, <-,>=latex, very thick] (3.6,-0.66) arc (-90:-25:0.97 and 0.135); 
\end{tikzpicture}
\end{center}
\caption{Motion planning in $D_{31}$}\label{d31}\end{figure}
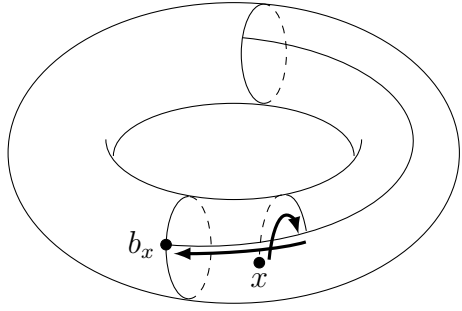
On the other hand, let $D_{32}$ consist of the pairs $(x, z)$ in (ii) for which there exists $y \in A_{x}\setminus\left(C_{x}^{I}\cup\left\{a_{x}\right\}\right)$ with $\pi(y)=z$. An argument identical to the one given in the case of $D_{2}$ shows that the element $y$ on the definition of $D_{32}$ is unique and depends
continuously on $\left(x, z\right) \in D_{32}$. Thus, a section $s_{32}: D_{32} \rightarrow P \T$ as the one we need sends $(x, z) \in D_{32}$
into the path in $\T$ from $x$ to $y$ depicted by the thick arrows in Figure~\ref{d32}. Note that the condition $x \in \A \cup \B$ implies that $a_x$ is the
intersection point (removed from Figure~\ref{d32}) of $C_x$ and $C_x^D$, which yields continuity on $(x,z)$ of $s_{32}$.
\begin{figure}\begin{center}
\begin{tikzpicture}
\draw(1,0) circle (1.5);
\draw(-6.6,0) .. controls (-3.2,-0.8) .. (-0.5,0);
\filldraw[color = black] (-4,0) circle (2pt) ;
\draw[dashed] (-4,0) arc (160:90:0.4 and 1.2) ;
\draw (-3.5,0.8) arc (90:-10:0.4 and 1.2) ;
\draw[very thick, ->,  >=latex] (-4.1,0.3) arc (160:-8:0.6 and 1.2) ;
\draw[very thick, ->, >=latex](-2.8 ,-0.4) .. controls (-1.7,-0.2) .. (-0.8,0.1);
\draw[very thick, <-, >=latex](-6.4,-0.23) .. controls (-3.8,-0.75) .. (-3.2 ,-0.73);
\draw[very thick, ->, >=latex] (0,0.5) arc (160:45:1);
\draw[very thick, ->, >=latex] (2.05,0) arc (10:-140:1);
\node (A) at (-4.3, 0.1) {$x$};
\node (B) at (-2.95, -0.45) {$y$};
\node (D) at (-6.6, -0.3) {$y$};
\node (C) at (2, 0.5) {$y$};
\node (D) at (0, -0.5) {$y$};
\filldraw[color = black, fill=white] (-0.5,0) circle (3pt) node {};
\filldraw[color = black, fill=white] (-6.7,0) circle (3pt) node {};
\end{tikzpicture}
\end{center}
\caption{Motion planning in $D_{32}$}\label{d32}\end{figure}
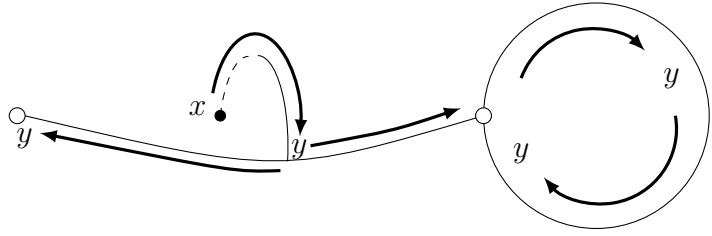
Since an element $(x, z) \in \overline{D_{31}}$ must have $z=\left[a_{x}\right],$ while an element $(x, z) \in \overline{D_{32}}$ must have $x \in
\A \cup \B$, it follows that
\[
\overline{D_{31}} \cap D_{32}=\varnothing=D_{31} \cap \overline{D_{32}}.
\]
Therefore $s_{31}$ and $s_{32}$ yield a section $s_{3}$ on $D_{3}:=D_{31} \cup D_{32} $ as the one we need.

\smallskip
The last domain is
\[
D_{4}:=\T \times \K-\left(D_{1} \cup D_{2} \cup D_{3}\right),
\]
i.e., $D_{4}$ consist of the pairs $\left(x, \left[a_{x}\right]\right) \in(\A \cup \B) \times \K$. As in the case of $D_{31}, D_{4}$
admits a section $s_4: D_4 \to P \T$ as the one we need. The proof of Theorem \ref{T} is now complete.

\section{Effectual motion planners on spheres}\label{seccionfinal}
We close the paper by describing an effectual motion planner on $S^n\subset\mathbb{R}^{n+1}$ with $n+2$ domains in the general case, and $n+1$ domains provided $n\in\{1,3,7\}$. We conjecture that these are optimal planners, i.e., that $\TC^{\mathbb{Z}_2}_{\effl}(S^n)=n+\delta_n$, where $\delta_n=1$ except for $\delta_1=\delta_3=\delta_7=0$.

Set $k=n+\delta_n$ and choose a continuous map $v=(v_0,v_1,\ldots,v_k)\colon S^n\to (S^n)^{k+1}$ with $v_0(p)=p$ and so that
\begin{equation}\label{generadores}
\mbox{$v_0(p),v_1(p),\ldots,v_k(p)$ generate $\mathbb{R}^{n+1}$ for each $p\in S^n$.}
\end{equation}
For instance, in the non-parallelizable case, $v_i(p)$ can be taken to be the $i$-th canonical basis element (for all $p$). For $p\in S^n$, let $H_i(p)=\{q\in\mathbb{R}^{n+1}\colon \langle q,v_i(p)\rangle=0\}$ and $H^+_i(p)=\{q\in\mathbb{R}^{n+1}\colon \langle q,v_i(p)\rangle>0\}$, where $\langle-,-\rangle$ denotes the standard inner product in $\mathbb{R}^{n+1}$. Recall that $\pi\colon S^n\to P^n$ stands for the projection. The sets $$D_i=\left\{(p,\ell)\in S^n \times P^n\colon \ell\in \pi\left(\left(\bigcap_{0\leq j< i}H_j(p)\right)\cap H_i^+(p)\cap S^n\right)\right\},\quad 0\leq i\leq k,$$ cover $S^n \times P^n$ in view of~(\ref{generadores}), so it suffices to construct a section for each restriction $\epsilon_{|D_i}$. Note that, for each $(p,\ell)\in D_i$, the intersection $\ell \cap S^n\cap H^+_i(p)$ reduces to a point $q(p,\ell)$ which is different from $-p$ and depends continuously on $(p,\ell)$. The required section $D_i\to PS^n$ sends $(p,\ell)$ into the shortest geodesic (at constant speed) from $p$ to $q(p,\ell)$.


\bigskip

\small
{\sc Departamento de Matem\'aticas

Centro de Investigaci\'on y de Estudios Avanzados del I.P.N.

Av.~Instituto Polit\'ecnico Nacional no.~2508, San Pedro Zacatenco

M\'exico City 07000, M\'exico}

\tt cadavid@math.cinvestav.mx

jesus@math.cinvestav.mx

\bigskip

{\sc Departamento de Formaci\'on B\'asica Disciplinaria

Unidad Profesional Interdisciplinaria de Ingenier\'ia Campus Hidalgo

Carretera Pachuca-Actopan Km.~1+500

Ciudad del Conocimiento y la Cultura 42162, Hidalgo}

\tt bgutierrezm@ipn.mx

\bigskip

{\sc Departamento de Matem\'atica

Universidade de S\~ao Paulo

Instituto de Ci\^encias Matem\'aticas e Computa\c{c}\~ao -- USP

Avenida Trabalhador S\~ao-carlense, 400 -- Centro

S\~ao Carlos 13566-590, Brazil}

\tt cesarzapata@usp.br

\end{document}